\documentclass[reqno]{amsart}
\usepackage{amsthm}
\usepackage{amsmath}
\usepackage{amssymb}
\usepackage{amsfonts}
\usepackage[margin=1.5in]{geometry}
\usepackage{verbatim}
\usepackage{hyperref}

\newtheorem{theorem}{Theorem}[section]
\newtheorem{definition}[theorem]{Definition}

\newtheorem{proposition}[theorem]{Proposition}

\newtheorem{lemma}[theorem]{Lemma}
\newtheorem{remark}[theorem]{Remark}

\DeclareMathOperator{\NS}{NS}
\DeclareMathOperator{\Pic}{Pic}
\DeclareMathOperator{\Div}{Div}

\DeclareMathOperator{\sgn}{sgn}

\author{Paul Reschke and Bar Roytman}
\date{}
\title{Lower Semi-Continuity of Entropy in a Family of K3 Surface Automorphisms}

\thanks{The first author was partially supported by NSF grants DMS-0943832 and DMS-1045119. The second author was partially supported by NSF grant DMS-1266207.}

\begin{document}

\begin{abstract}
We compute topological entropies for a large family of automorphisms of K3 surfaces in \(\mathbb{P}^1 \times \mathbb{P}^1 \times \mathbb{P}^1\). In keeping with a result by Xie \cite{Xie}, we find that the entropies vary in a lower semi-continuous manner as the Picard ranks of the K3 surfaces vary.
\end{abstract}

\maketitle

\section{Introduction}

We compute entropies in a family of automorphisms of complex K3 surfaces in
\[\mathbb{P}^1 \times \mathbb{P}^1 \times \mathbb{P}^1 = \{(x=[x_0:x_1],y=[y_0:y_1],z=[z_0:z_1])\}.\]
The set of all effective divisors on \(\mathbb{P}^1 \times \mathbb{P}^1 \times \mathbb{P}^1\) of tri-degree \((2,2,2)\) is parametrized by \(\mathbb{P}^{26}\), and every non-singular prime divisor in this set is a K3 surface; so a general effective divisor of tri-degree \((2,2,2)\) is a K3 surface. Throughout this paper, \(Q=Q(x_0,x_1,y_0,y_1,z_0,z_1)\) is a tri-homogeneous polynomial of tri-degree \((2,2,2)\) and \(S\) is a K3 surface in \(\mathbb{P}^1 \times \mathbb{P}^1 \times \mathbb{P}^1\) of the form \(\{Q=0\}\).

We write
\[Q(x_0,x_1,y_0,y_1,z_0,z_1) = \sum_{j \in \{0,1,2\}} x_0^j x_1^{2-j}Q_{x,j}(y_0,y_1,z_0,z_1)\]
(so each non-trivial \(Q_{x,j}=Q_{x,j}(y_0,y_1,z_0,z_1)\) is bi-homogeneous of bi-degree \((2,2)\)), and for irreducible \(Q\) we define a birational involution \(\tau_x\) on \(\mathbb{P}^1 \times \mathbb{P}^1 \times \mathbb{P}^1\) by
\[\tau_x(x,y,z) = ([x_0Q_{x,2}+x_1Q_{x,1}: -x_1Q_{x,2}],y,z).\]
For \((x,y,z) \in S\) in the domain of \(\tau_x\),
\[\tau_x(x,y,z) = ([x_1Q_{x,0}:x_0Q_{x,2}],y,z) \in S;\]
since \(S\) is its own unique minimal model, it follows that \(\tau_x\) defines an automorphism of \(S\). We define \(\tau_y\) and \(\tau_z\) similarly; so \(\mathrm{Aut}(S)\) contains the subgroup generated by \(\{\tau_x,\tau_y,\tau_z\}\).

Silverman and Mazur \cite{Maz} first suggested compositions of the involutions just described as interesting examples of infinite-order automorphisms of K3 surfaces. Wang \cite{Wan} and Baragar \cite{Ba1} used automorphisms in this subgroup to study rational points on \(S\) (when \(S\) is defined over a number field). Cantat \cite{Can} and McMullen \cite{McM} highlighted \(f:=\tau_z \circ \tau_y \circ \tau_x\) on various choices of \(S\) as examples of K3 surface automorphisms with positive topological entropy. Cantat observed that results by Gromov \cite{Gro}, Yomdin \cite{Yom}, and Friedland \cite{Fri} imply that the entropy of \(f\) is the logarithm of the spectral radius \(\lambda(f)\) of \(f^*:\Pic(S) \rightarrow \Pic(S)\). Wang, Cantat, and McMullen showed how to compute \(f^*\) in the very general case where \(S\) has Picard rank \(\rho(S)=3\). Baragar \cite{Ba2} showed how to compute \(f^*\) in a special family where \(\rho(S)=4\), and thereby showed that \(\lambda(f)\) is not constant among all \(K3\) surfaces in \(\mathbb{P}^1 \times \mathbb{P}^1 \times \mathbb{P}^1\). Here, we compute \(f^*\) for a much larger set of choices of \(S\), with \(\rho(S)\) ranging from 3 to 11.

For all \(p \in \mathbb{P}^1\), we let \(E_{x=p}\) (resp., \(E_{y=p}\), \(E_{z=p}\)) denote the restriction to \(S\) of the prime divisor \(\{x=p\}\) (resp., \(\{y=p\}\), \(\{z=p\}\)) on \(\mathbb{P}^1 \times \mathbb{P}^1 \times \mathbb{P}^1\); we call each \(E_{x=p}\) (resp., \(E_{y=p}\), \(E_{z=p}\)) a fiber of \(S\) over the \(x\)-axis (resp., \(y\)-axis, \(z\)-axis). Each fiber is an effective divisor of bi-degree \((2,2)\) in \(\mathbb{P}^1 \times \mathbb{P}^1\), and hence is an elliptic curve if it is a non-singular prime divisor; so a general fiber is an elliptic curve.

For all \(p=(p_1,p_2) \in \mathbb{P}^1 \times \mathbb{P}^1\), we define (in \(\mathbb{P}^1 \times \mathbb{P}^1 \times \mathbb{P}^1\)) \(C_{x,p} := \{y=p_1\} \cap \{z=p_2\}\), \(C_{y,p} := \{x=p_2\} \cap \{z=p_1\}\), and \(C_{z,p} := \{x=p_1\} \cap \{y=p_2\}\); we call each \(C_{x,p}\) (resp., \(C_{y,p}\), \(C_{z,p}\)) a curve parallel to the \(x\)-axis (resp., \(y\)-axis, \(z\)-axis). It may happen that \(S\) contains a curve parallel to an axis. If, for example, \(C_{x,p} \subseteq S\), then neither \(E_{y=p_1}\) nor \(E_{z=p_2}\) is a prime divisor.

For a divisor \(D\) on \(S\), we let \([D]\) denote the class of \(D\) in \(\Pic(S)\). We let \((\_ \cdot \_)\) denote the intersection form on both \(\Pic(S)\) and \(\Div(S)\). In light of the fact that the fibers of \(S\) over a fixed axis are all linearly equivalent, we let \(E_x\), \(E_y\) and \(E_z\) in \(\Pic(S)\) denote the classes of the fibers over, respectively, the \(x\)-, \(y\)-, and \(z\)-axes. We let \(\mathcal{B}_x(S)\), \(\mathcal{B}_y(S)\), and \(\mathcal{B}_z(S)\) denote the sets of all classes of curves parallel to, respectively,  the \(x\)-, \(y\)-, and \(z\)-axes which are contained in \(S\), and we set
\[\mathcal{B}(S) := \{E_x,E_y,E_z\} \cup \mathcal{B}_x(S) \cup \mathcal{B}_y(S) \cup \mathcal{B}_z(S).\]
Since \(K_S\) is trivial, the adjunction formula gives \((E_\omega \cdot E_\omega)=0\) for each \(E_\omega\) and \((C \cdot C) = -2\) for each curve \(C \subseteq S\) parallel to an axis; it follows that the number of distinct classes in \(\mathcal{B}(S)\) is 3 plus the number of distinct curves parallel to axes in \(S\).

\begin{definition}\label{PureType}
For an ordered triple \((k,l,m)\) of non-negative integers, we say that \(S\) is ``pure of type \((k,l,m)\)'' if the following conditions hold:
\begin{itemize}
\item[a)] \(|\mathcal{B}_x(S)|=k\), \(|\mathcal{B}_y(S)|=l\), and \(|\mathcal{B}_z(S)|=m\);
\item[b)] \(\mathcal{B}(S)\) is a basis for \(\Pic(S)\); and
\item[c)] \((\mathcal{L} \cdot \mathcal{L}') = 0\) whenever \(\mathcal{L}\) and \(\mathcal{L}'\) are distinct classes in \(\mathcal{B}_x(S) \cup \mathcal{B}_y(S) \cup \mathcal{B}_z(S)\).
\end{itemize}
\end{definition}

We let \(\mathcal{U}_{k,l,m} \subseteq \mathbb{P}^{26}\) denote the set of all K3 surfaces which are pure of type \((k,l,m)\). If \((k',l',m')\) is a reordering of \((k,l,m)\), then \(\mathcal{U}_{k',l',m'} \cong \mathcal{U}_{k,l,m}\). If \(S \in \mathcal{U}_{k,l,m}\), then the conditions in Definition \ref{PureType} provide sufficient information for the computation of \(f^*\). However, it is a significant step to show that pure K3 surfaces of various types even exist. For distinct ordered triples \((k,l,m)\) and \((k',l',m')\), we write \((k,l,m) < (k',l',m')\) if \(k \leq k'\), \(l \leq l'\), and \(m \leq m'\). We set
\[\mathcal{N}'' := \{(6,0,0),(5,1,1),(4,2,2),(3,3,3)\},\]
we let \(\mathcal{N}'\) denote the set of all permutations of ordered triples in \(\mathcal{N}''\), and we let \(\mathcal{N}\) denote the set of all ordered triples \((k,l,m)\) of non-negative integers satisfying \((k,l,m) \leq \nu\) for some \(\nu \in \mathcal{N}'\).

\begin{theorem}\label{Existence}
For \((k,l,m) \in \mathcal{N}-\{(3,3,3)\}\), the dimension of the space of isomorphism classes of K3 surfaces contained in \(\mathcal{U}_{k,l,m}\) is \(17-k-l-m\). If \((k',l',m') \in \mathcal{N}\) satisfies \((k,l,m) < (k',l',m')\), then \(\mathcal{U}_{k',l',m'}\) is contained in the closure of \(\mathcal{U}_{k,l,m}\).

For \((k,l,m) \notin \mathcal{N}\), \(\mathcal{U}_{k,l,m} = \emptyset\).
\end{theorem}

We prove Theorem \ref{Existence} in \S \ref{FIND}. The proof relies on the surjectivity of the period map for K3 surfaces to show the existence of \(S \in \mathcal{U}_{k,l,m}\), and thus does not yield any explicit equations defining pure K3 surfaces in \(\mathbb{P}^1 \times \mathbb{P}^1 \times \mathbb{P}^1\). Baragar and van Luijk \cite{BvL} have given explicit equations for some pure K3 surfaces of type \((0,0,0)\), and Barager \cite{Ba2} has given explicit equations for some pure K3 surfaces of type \((1,0,0)\). Little else in the direction of concrete examples has appeared in the literature, and it is typically quite challenging to show that a particular polynomial \(Q\) defines a pure K3 surface. We do not know if \(\mathbb{P}^1 \times \mathbb{P}^1 \times \mathbb{P}^1\) contains pure K3 surfaces of type \((3,3,3)\).

Theorem \ref{Existence} shows that we can compute and compare entropies among many different types of K3 surface automorphisms even by focusing only on automorphisms of pure K3 surface automorphisms.

\begin{theorem}\label{Entropies}
As \(S\) varies among all pure K3 surfaces, \(\lambda(f)\) depends only on the type of \(S\). Writing \(\lambda(f)=\lambda(k,l,m)\) as a function of the type of \(S\), we have \(\lambda(k,l,m) > \lambda(k',l',m')\) whenever \((k,l,m) < (k',l',m')\).
\end{theorem}

We prove Theorem \ref{Entropies} in \S \ref{COMPUTE} by computing \(\lambda(f)\) for every pure K3 surface. We note that \(\lambda(f)\) actually depends only on the unordered triple \((k,l,m)\); that is, \(\lambda(k',l',m')=\lambda(k,l,m)\) if \((k',l',m')\) is a reordering of \((k,l,m)\). However, the computation of \(f^*\) does depend on the order of \((k,l,m)\). We compute \(\lambda(3,3,3) = 1\), which suggests that \(f\) has some very special behavior on pure K3 surfaces of type \((3,3,3)\) if any exist (and so perhaps suggests that such K3 surfaces should not exist).

Theorems \ref{Existence} and \ref{Entropies} show that \(\lambda(f)\) is a strictly lower semi-continuous function of the parameters in the union of all of the spaces \(\mathcal{U}_{k,l,m}\). Thus the set of all pure K3 surfaces provides an example that demonstrates the following result of Xie.

\begin{theorem}[\cite{Xie}, Theorem 4.3]\label{Xie}
Suppose \(W\) is a quasi-projective variety, \(\mathcal{S} \rightarrow W\) is a family of projective surfaces, and \(F\) is an automorphism of \(\mathcal{S}\) that preserves each fiber over \(W\). For \(s \in W\), let \(h(s)\) denote the entropy of the restriction of \(F\) to the fiber over \(s\). Then \(h\) is a lower semi-continuous function on \(W\).
\end{theorem}

(In \cite{Xie}, Theorem \ref{Xie} is phrased more generally with entropies replaced by first dynamical degrees, in which case \(F\) and its restrictions need only be birational self-maps rather than automorphisms.)

Theorem \ref{Xie} applies to \(\lambda(f)\) in the following way: \(\mathbb{P}^{26} \times \mathbb{P}^1 \times \mathbb{P}^1 \times \mathbb{P}^1\) admits a birational involution that restricts to \(f\) on every fiber \(\mathbb{P}^1 \times \mathbb{P}^1 \times \mathbb{P}^1\) of the projection to \(\mathbb{P}^{26}\) where \(f\) is well-defined; this involution preserves the variety in \(\mathbb{P}^{26} \times \mathbb{P}^1 \times \mathbb{P}^1 \times \mathbb{P}^1\) defined by \(Q(x_0,x_1,y_0,y_1,z_0,z_1)=0\), and hence realizes most quasi-subvarieties of \(\mathbb{P}^{26}\) as paramater spaces for families of K3 surface automorphisms of the sort treated in Theorem \ref{Xie}.

Although pure K3 surfaces are very general among all K3 surfaces \(S \subseteq \mathbb{P}^1 \times \mathbb{P}^1 \times \mathbb{P}^1\), they certainly do not account for all \(S\). One could adapt the procedure in this paper to the computation of \(\lambda(f)\) among all \(S\) satisfying (a) and (b) but not necessarily (c) in Definition \ref{PureType}, since \(\Pic(S)\) and \(f^*\) can still be sufficiently well understood for such \(S\); the challenge then would be to determine which arrangements of curves parallel to axes actually occur on such \(S\). Even so, as first observed by Rowe \cite{Row}, a K3 surface \(S\) can fail even to satisfy (b), in which case it is impossible to compute \(\lambda(f)\) in the manner used here without some means of determing \(\Pic(S)\); the K3 surface \(\tilde{S} \subseteq \mathbb{P}^1 \times \mathbb{P}^1 \times \mathbb{P}^1\) below is an example which fails to satisfy (b).
\\
\\
\noindent \textbf{Acknowledgements.} We thank Mattias Jonsson for providing the opportunity to do some of the calculations in this paper as part of an REU project. We thank Mattias Jonsson and Serge Cantat for suggestions which helped to clarify Theorem \ref{Xie}.

\section{Finding Pure K3 Surfaces}\label{FIND}

Every prime divisor on \(\mathbb{P}^1 \times \mathbb{P}^1 \times \mathbb{P}^1\) is the zero locus of an irreducible tri-homogeneous polynomial (and every such zero locus is a prime divisor). The classes of \(\{x_0=0\}\), \(\{y_0=0\}\), and \(\{z_0=0\}\) generate \(\Pic(\mathbb{P}^1 \times \mathbb{P}^1 \times \mathbb{P}^1)\). It is a well-known fact (e.g.,\cite{Maz},\cite{Wan},\cite{McM}) that every smooth prime divisor \(S\) of tri-degree \((2,2,2)\) is a K3 surface; one may verify this by using the Lefschetz hyperplane theorem (applied to \(S\) as a hyperplane section of \(\mathbb{P}^1 \times \mathbb{P}^1 \times \mathbb{P}^1\)) to show that \(h^1(S)=0\) and using the adjunction formula (applied to \(S\) as a divisor on \(\mathbb{P}^1 \times \mathbb{P}^1 \times \mathbb{P}^1\)) to show that \(K_S\) is trivial.

\begin{lemma}\label{notK3}
Let \(S'\) be a smooth prime divisor on \(\mathbb{P}^1 \times \mathbb{P}^1 \times \mathbb{P}^1\) of tri-degree \((a,b,c)\). If \(abc > 0\) and \((a,b,c) \neq (2,2,2)\), then \(S'\) is neither a K3 surface nor a copy of \(\mathbb{P}^2\) nor a Hirzebruch surface. If \(abc=0\), then \(S'\) is a product with one of the coordinate copies of \(\mathbb{P}^1\) as factor.
\end{lemma}

\begin{proof}
First suppose \(abc > 0\) and \((a,b,c) \neq (2,2,2)\). The effective divisors
\[D_1 := \{x_0=0\}|_{S'}, D_2 := \{y_0=0\}|_{S'}, \text{ and } D_3 := \{z_0=0\}|_{S'}\]
all satisfy \((D_j \cdot D_j) = 0\) and \((D_j \cdot D_{j' \neq j}) > 0\). Thus \(\{[D_1],[D_2],[D_3]\}\) is a linearly independent set in \(\Pic(S')\). By the adjunction formula,
\[K_{S'} = (a-2)[D_1] + (b-2)[D_2] + (c-2)[D_3]\]
--which is not trivial. So \(S'\) is not a K3 surface. Also, \(\rho(S') \geq 3\) implies that \(S'\) is neither a copy of \(\mathbb{P}^2\) nor a Hirzebruch surface.

If \(abc=0\), the claim is evident from the form of the polynomial defining \(S'\).
\end{proof}

A lattice of rank \(r \in \mathbb{N}\) is a group \(L \cong \mathbb{Z}^r\) equipped with a bilinear form \((\_ \cdot \_)_L\) which is integral, symmetric, and non-degenerate. Given a basis for \(L\), there is a unique integer matrix \(M\) such that \((\vec{g}_1 \cdot \vec{g}_2)_L = \vec{g}_1^{\, \, t} M \vec{g}_2\) for all \(\vec{g}_1,\vec{g}_2 \in L\). Since \(M\) is symmetric with \(\det(M) \neq 0\), its eigenvalues are all non-zero real numbers. The signature of \(L\) is \((p,q)\), where \(p\) and \(q\) denote the number (counting multiplicity) of, respectively, positive and negative eigenvalues of \(M\). If \(T\) is a projective K3 surface, it is a well-known consequence of the Hodge index theorem (e.g.,\cite{bhpv}) that the intersection form makes \(\Pic(T) \cong \NS(T)\) into a lattice of signature \((1,\rho(T)-1)\).

For every K3 surface \(S \subseteq \mathbb{P}^1 \times \mathbb{P}^1 \times \mathbb{P}^1\), the intersection form on \(\langle E_x,E_y,E_z \rangle \leq \Pic(S)\) is given by
\[M_{0,0,0} := \left( \begin{array}{ccc}
0 & 2 & 2 \\
2 & 0 & 2 \\
2 & 2 & 0
\end{array} \right).\]
For every ordered triple \((k,l,m)\) of non-negative integers, the conditions in Definition 1.1 indicate how to write a matrix \(M_{k,l,m}\) that gives the intersection form on \(\Pic(S)\) in the basis \(\mathcal{B}(S)\) whenever \(S\) is pure of type \((k,l,m)\). For example,
\[M_{2,0,1} = \left( \begin{array}{cccccc}
0 & 2 & 2 & 1 & 1 & 0 \\
2 & 0 & 2 & 0 & 0 & 0 \\
2 & 2 & 0 & 0 & 0 & 1 \\
1 & 0 & 0 & -2 & 0 & 0 \\
1 & 0 & 0 & 0 & -2 & 0 \\
0 & 0 & 1 & 0 & 0 & -2
\end{array} \right).\]

\begin{lemma}\label{PureDet}
For any ordered triple \((k,l,m)\) of non-negative integers,
\[\det(M_{k,l,m}) = -(-2)^{k+l+m-3}(128-16(k+l+m)+klm).\]
\end{lemma}

\begin{proof}
The formula given follows by computation from the general formula
\[\det((a_{i,j})_{1 \leq i,j \leq n}) = \sum \sgn(\xi) \prod_{i=1}^n a_{i,\xi(i)},\]
where the sum is taken over all permutations \(\xi\) of \(\{1,\dots,n\}\).
\end{proof}

For \((k,l,m)\) such that \(\det(M_{k,l,m}) \neq 0\), which includes all \((k,l,m) \in \mathcal{N}\), let \(L_{k,l,m}\) denote the lattice given by \(M_{k,l,m}\). If \((k',l',m')\) is a reordering of \((k,l,m)\), then \(L_{k',l',m'}\) is isometric to \(L_{k,l,m}\).

For any K3 surface \(T\), the Riemann-Roch theorem and the adjunction formula imply the following useful facts about the intersection form on \(\Pic(T)\) (e.g.,\cite{May},\cite{Dol},\cite{bhpv}):
\begin{itemize}
\item[\(\bullet\)] if \(\mathcal{L} \in \Pic(T)\) satisfies \((\mathcal{L} \cdot \mathcal{L}) \geq -2\), then either \(\mathcal{L}\) is effective or \(-\mathcal{L}\) is effective;
\item[\(\bullet\)] if \(\mathcal{L} \in \Pic(T)\) is effective, then \(h^0(\mathcal{L}) \geq 2 + \frac{1}{2}(\mathcal{L} \cdot \mathcal{L})\);
\item[\(\bullet\)] if \(D \in \Div(T)\) is reduced, effective, and connected, then \(h^0([D]) = 2 + \frac{1}{2}(D \cdot D)\);
\item[\(\bullet\)] if \(D\) is a prime divisor on \(T\), then \(h^0([D]) \geq -2\).
\end{itemize}

\subsection{Global sections in pure Picard lattices}

Fix an ordered triple \((k,l,m)\) of non-negative integers, and suppose \(T\) is a K3 surface such that \(\Pic(T)\) is isometric to \(L_{k,l,m}\). (It is then implicit here that \(\det(M_{k,l,m} \neq 0)\).) Since \(L_{k,l,m}\) contains elements with positive self-intersection, it follows from Grauert's criterion (e.g.,\cite{bhpv}) that \(T\) is projective. Let
\[\mathcal{B} = \{B_1,B_2,B_3,B_{x,1},\dots,B_{x,k},B_{y,1},\dots,B_{y,l},B_{z,1},\dots,B_{z,m}\}\]
be a basis for \(\Pic(T)\) in which \(M_{k,l,m}\) gives the intersection form, and suppose further that each \(B_j\) is nef. For \((k,l,m) \in \mathcal{N}\), we will show that there is an embedding \(T \subseteq \mathbb{P}^1 \times \mathbb{P}^1 \times \mathbb{P}^1\) as a pure K3 surface of type \((k,l,m)\).

If for some \(B_j\) there were \(\mathcal{L} \in \langle B_j \rangle^\perp \leq \Pic(T)\) satisfying \((\mathcal{L} \cdot \mathcal{L})=0\) and \(\mathcal{L} \notin \langle B_j \rangle\), then \(\langle B_j,\mathcal{L} \rangle\) would be a totally isotropic sublattice of \(\Pic(T)\) of rank 2; but it is a well-known fact (e.g.,\cite{McM}) that the signature of \(\Pic(T)\) implies that \(\Pic(T)\) cannot contain a totally isotropic sublattice of rank \(r>1\). It follows that each \(\langle B_j \rangle^\perp\) is negative definite away from \(\langle B_j \rangle\). One can check by computation that every \(\mathcal{L} \in \langle B_1,B_2,B_3, \rangle^\perp\) satisfies \((\mathcal{L} \cdot \mathcal{L}) \equiv 0 \bmod 4\)--so that, in particular, \(\langle B_1,B_2,B_3 \rangle^\perp\) cannot contain the class of any prime divisor on \(T\).

\begin{lemma}\label{PrimeBasis}
Every element of \(\mathcal{B}\) is the class of a prime divisor on \(T\).
\end{lemma}

\begin{proof}
Since \((B_{x,1} \cdot B_{x,1}) = -2\) and \((B_1 \cdot B_{x,1}) = 1\) (assuming \(k>0\)), \(B_{x,1}\) must be effective. Write \(B_{x,1} = [D_1] + \dots + [D_n]\), where each \(D_j\) is a prime divisor (but the prime divisors may not be pairwise distinct a priori). Since \(B_{x,1} \in \langle B_2,B_3 \rangle^\perp\), \((B_{x,1} \cdot B_1)=1\), and no \(D_j\) can have its class in \(\langle B_1,B_2,B_3 \rangle^\perp\), the only possibility is \(n=1\)--so that \(B_{x,1}\) is the class of a prime divisor. It follows similarly that each \(B_{\omega,j}\) is the class of a prime divisor \(D_{\omega,j}\).

We will now show that \(B_1\) is the class of a prime divisor. It will follow similarly that \(B_2\) and \(B_3\) are classes of prime divisors. Each \(B_j\) is effective with \(h^0(B_j) \geq 2\) since it is nef and satisfies \((B_j \cdot B_j) = 0\).

First suppose \(l=m=0\). In this case, \((\mathcal{L} \cdot \mathcal{L}) \equiv 0 \bmod 4\) whenever \(\mathcal{L} \in \langle B_1 \rangle^\perp\). Also, \((B_2 \cdot \mathcal{L})\) and \((B_3 \cdot \mathcal{L})\) are even for every \(\mathcal{L} \in \Pic(T)\). It then follows from the intersection numbers given by \(M_{0,0,0}\) that \(B_1\) cannot be written as a sum of more than one class of a prime divisor.

Now suppose \(l>0\); the case \(m>0\) follows similarly. Since \(B_1' := B_1 - B_{y,1}\) satisfies \((B_1' \cdot B_1')=-2\) and \((B_1' \cdot B_2)=1\), it is effective. Write \(B_1' = [D_1] + \dots + [D_n]\), where each \(D_j\) is a prime divisor (but the prime divisors may not be pairwise distinct a priori). The intersection numbers of \(B_1'\) with \(B_1\), \(B_2\), and \(B_3\) force \(n \leq 3\) and \((D_j \cdot D_j) = -2\) for each \(D_j\). Moreover, there is a unique \(D_j\) satisfying \(([D_j] \cdot B_2) >0\); take \(D_1\) to be this divisor--so that \(([D_1] \cdot B_2)=1\) and \(([D_j] \cdot B_3) > 0\) for \(j>1\).

If \(n=1\), then \((D_1 \cdot D_{y,1})=2\). If \(n=2\), then \((B_1' \cdot B_1') = -2\) implies \((D_1 \cdot D_2)=1\). If \(D_1 \in \mathcal{B}\), then \(([D_1] \cdot B_1') \in \{0,2\}\) gives a contradiction. So, since \(B_{y,1}+[D_1]\) and \(B_{y,1}+[D_2]\) are both in \(\langle B_1 \rangle^\perp\), \((B_{y,1} \cdot B_1')=2\) implies \((D_1 \cdot D_{y,1}) = (D_2 \cdot D_{y,1}) = 1\).

If \(n=3\), then \(([D_2] \cdot B_3) = ([D_3] \cdot B_3) = 1\) and \(([D_1] \cdot B_3)=0\). If \(D_1=D_{y,1}\), then \((B_1' \cdot B_{y,1})=2\) and \((B_1' \cdot B_1')=-2\) force \(D_3=D_2\). If conversely \(D_3=D_2\), then \((B_1' \cdot B_1')=-2\) implies \((D_1 \cdot D_2) =2\)--so that \(\langle B_1,[D_1+D_2] \rangle\) is totally isotropic; since \((B_1 \cdot B_2)=2\) and \(([D_1+D_2] \cdot B_2)=1\), it follows that \(B_1 = 2[D_1]+2[D_2]\) and \(D_1=D_{y,1}\). So \(D_1=D_{y,1}\) if and only if \(D_2=D_3\); but then \((1/2)B_1 \in \Pic(T)\) is a contradiction in this case. Thus \([D_2+D_3] \in \langle B_1,B_2 \rangle^\perp\) and \([D_2-D_3] \in \langle B_1,B_2,B_3 \rangle^\perp\) imply \((D_2 \cdot D_3)=0\); also, by similar reasoning, \((D_1 \cdot D_{y,1})=0\). Since \((1/2)B_1 \notin \Pic(T)\) and \(\Pic(T)\) cannot contain a totally isotropic sublattice of rank 2, none of \((D_1 \cdot D_2)\), \((D_1 \cdot D_3)\), \((D_{y,1} \cdot D_2)\), or \((D_{y,1} \cdot D_3)\) can equal 2. So \(([D_1] \cdot B_1')=0\) and \((B_{y,1} \cdot B_1')=2\) imply
\[(D_1 \cdot D_2)=(D_1 \cdot D_3)=(D_{y,1} \cdot D_2)=(D_{y,1} \cdot D_3)=1.\]

In all three cases for \(n\), \(B_1\) is realized as the class of a reduced, effective, and connected divisor \(E\) with the property that every effective divisor \(E'\) satisfying \(E' < E\) has \(h^0(E')=1\). Fix \(\{s,s'\} \subseteq H^0(B_1)\) such that \(s\) vanishes on all of \(E\) and \(s'\) does not. If \(s'\) vanishes on some non-trivial effective divisor \(E'\) satisfying \(E' < E\), then \(h^0(E-E')=1\) contradicts the fact that \(s'/s\) is not constant. So \(B_1\) has no fixed component, and Proposition 1 in \cite{May} shows that \(B_1\) is the class of an elliptic curve.
\end{proof}

\begin{proposition}\label{K3embed}
There is an embedding \(T \subseteq \mathbb{P}^1 \times \mathbb{P}^1 \times \mathbb{P}^1\). If \((k,l,m) \in \mathcal{N}\), then \(T\) is pure of type \((k,l,m)\).
\end{proposition}

\begin{proof}
By Lemma \ref{PrimeBasis}, each \(B_j\) satisfies both \(h^0(B_j)=2\) and \((B_j \cdot B_j)=0\), and furthermore has no fixed component. Thus each \(B_j\) induces a morphism \(\psi_j:T \rightarrow \mathbb{P}^1\). Set
\[\psi := \psi_1 \times \psi_2 \times \psi_3:T \rightarrow \mathbb{P}^1 \times \mathbb{P}^1 \times \mathbb{P}^1\]
and \(A := B_1 + B_2 + B_3\), and let \(\phi\) denote the Segre embedding of \(\mathbb{P}^1 \times \mathbb{P}^1 \times \mathbb{P}^1\) into \(\mathbb{P}^7\). Then \(A = (\phi \circ \psi)^*\mathcal{O}(1)\). Since each \(B_j\) is nef and no prime divisor on \(T\) can have its class in \(\langle B_1,B_2,B_3 \rangle^\perp\), Nakai's criterion (e.g.,\cite{bhpv}) implies \(A\) is ample; also, \(A\) has no fixed component since no \(B_j\) does. So \((\phi \circ \psi)\) does not collapse any curve on \(T\) and Proposition 2 in \cite{May} shows that \((\phi \circ \psi)\) is either an embedding or a ramified double covering. Thus \(\psi(T)\) is a prime divisor on \(\mathbb{P}^1 \times \mathbb{P}^1 \times \mathbb{P}^1\). Since each \(B_j+B_{j' \neq j}\) is nef, big, and effective with no fixed component, Proposition 2 in \cite{May} shows also that each \(\psi_j \times \psi_{j'}\) is surjective. So, in particular, \(\psi(T)\) is not a product with one of the coordinate copies of \(\mathbb{P}^1\) as a factor. If \((\phi \circ \psi)\) is a ramified double covering, then the main result in \cite{Rei} shows that \(\psi(T)\) is either a copy of \(\mathbb{P}^2\) or a Hirzebruch surface--which contradicts Lemma \ref{notK3}. Thus \(\psi\) is an embedding.

For each \(B_{\omega,j'}\) and each \(B_j\) with \((B_j \cdot B_{\omega,j'})=0\), \(h^0(B_j)\) must contain a section whose zero locus is disjoint from \(D_{\omega,j'}\)--which means \(\psi_j(D_{\omega,j'})\) is point. Thus each \(\psi(D_{\omega,j'})\) is a curve parallel to an axis (specifically, the axis corresponding to the \(B_j\) which satisfies \((B_j \cdot B_{\omega,j'})=1\)) and \(\psi(T)\) is pure of type \((k,l,m)\) if it has no curves parallel to axes beyond those whose classes are contained in \(\mathcal{B}\).

Now consider the case \((k,l,m) \in \mathcal{N}\). Suppose \(\psi(T)\) contains some \(C_{x,p}\) with \([C_{x,p}] \notin \mathcal{B}\). By the construction of \(\psi\), \(([C_{x,p}] \cdot B_1) = 1\) and \([C_{x,p}]\) must have intersection zero with \(B_2\), \(B_3\), and every \(B_{x,j}\). If \([C_{x,p}]\) has intersection zero with every \(B_{y,j}\) and every \(B_{z,j}\), then the intersection form on \(\langle \mathcal{B} \cup \{[C_{x,p}]\} \rangle\) is given by \(M_{k+1,l,m}\); but then Lemma \ref{PureDet} shows \(\mathcal{B} \cup \{[C_{x,p}]\}\) is linearly independent, a contradiction. Writing \(p = (p_1,p_2)\) and \(p'=(p_1',p_2')\), every curve \(C_{y,p'}\) satisfies \(C_{y,p'} \cap C_{x,p} = \emptyset\) if \(p_1' \neq p_2\) and \(|C_{y,p'} \cap C_{x,p}| = 1\) with multiplicity 1 if \(p_1' = p_2\). Since \(E_{z=p_2}\) has bi-degree \((2,2)\), there are at most two \(D_{y,j}\) on \(T\) such that \((C_{x,p} \cdot D_{y,j})=1\). If \((C_{x,p} \cdot D_{y,j'})=1\) for some \(D_{y,j'}\), then \((C_{x,p} \cdot D_{y,j})\) is odd--and hence equal to 1--for every \(D_{y,j}\); so \(l \leq 2\) in this case. Similarly, \(m \leq 2\) and \((C_{x,p} \cdot D_{z,j})=1\) for every \(D_{z,j}\) if there is some \(D_{z,j'}\) such that \((C_{x,p} \cdot D_{z,j'})=1\). One can now compute \(\det(M) \neq 0\) for each matrix \(M\) which gives a possible intersection form on \(\langle \mathcal{B} \cup \{[C_{x,p}]\} \rangle\), a contradiction. It would similarly be a contradiction if \(T\) contained some curve \(C_{y,p}\) or \(C_{z,p}\) whose class was not in \(\mathcal{B}\).
\end{proof}

\begin{remark}\label{NoD2}
Proposition \ref{K3embed} shows that \(n=1\) is the only case that can actually occur in the latter part of the proof of Lemma \ref{PrimeBasis} when \((k,l,m) \in \mathcal{N}\): otherwise, \(\psi(D_2)\) would be a curve parallel to the \(z\)-axis such that \([D_2] \notin \mathcal{B}\) (since \((D_2 \cdot D_{y,1})=1\)).
\end{remark}

\subsection{Nef classes in pure Picard lattices}

Fix \((k,l,m) \in \mathcal{N}\), set
\[\Gamma := \{\gamma \in L_{k,l,m} | (\gamma \cdot \gamma)_{L_{k,l,m}} = -2\},\]
and write \(\Gamma = \Gamma^+ \cup \Gamma^-\) such that \(\Gamma^+ \cap \Gamma^- = \emptyset\), \(\Gamma^- = \{\gamma|-\gamma \in \Gamma^+\}\), and
\[\Gamma \cap \{\gamma + \gamma'| \{\gamma,\gamma'\} \subset \Gamma^+\} \subseteq \Gamma^+;\]
we will call a choice of \(\Gamma^+\) satisfying these conditions ``allowable''. Let \(\mathcal{B}\) as above be a basis for \(L_{k,l,m}\) in which \(M_{k,l,m}\) gives the intersection form. We will show that \(\Gamma^+\) can be chosen so that each \(B_j\) satisfies \((B_j \cdot \gamma) \geq 0\) for every \(\gamma \in \Gamma^+\). Thus any effective isometry between \(L_{k,l,m}\) and the Picard lattice of a K3 surface--that is, any isometry which sends each \(\gamma \in \Gamma^+\) to an effective class--will send each \(B_j\) to a nef class.

Set
\[\tilde{Q}(x_0,x_1,y_0,y_1,z_0,z_1) := (x_0^2+x_1^1)(y_0^2+y_1^2)(z_0^2+z_1^2) + 3x_0x_1y_0y_1z_0z_1 -2x_1^2y_0y_1z_0z_1,\]
and set \(\tilde{S} := \{\tilde{Q}=0\} \subseteq \mathbb{P}^1 \times \mathbb{P}^1 \times \mathbb{P}^1\). One can check by directly testing possible factors that
\[\tilde{Q}(0,1,y_0,y_1,z_0,z_1)=y_0^2z_0^2+y_0^2z_1^2+y_1^2z_0^2-2y_0y_1z_0z_1+y_1^2z_1^2\]
is irreducible over \(\mathbb{C}\). So, since it has no factor of tri-degree \((1,0,0)\), \(\tilde{Q}\) is irreducible over \(\mathbb{C}\). It will also follow from Lemma \ref{SpecialK3} that \(\tilde{Q}\) is irreducible, since the existence of non-constant \(Q_1\) and \(Q_2\) satisfying \(Q_1 \cdot Q_2 = \tilde{Q}\) would imply \(\{Q_1 = Q_2 = 0\} \neq \emptyset\).

\begin{lemma}\label{SpecialK3}
The set
\[\mathrm{Sing}(\tilde{Q}) := \left\{ \tilde{Q} = \frac{\partial \tilde{Q}}{\partial x_0} = \frac{\partial \tilde{Q}}{\partial x_1} = \frac{\partial \tilde{Q}}{\partial y_0} = \frac{\partial \tilde{Q}}{\partial y_1} = \frac{\partial \tilde{Q}}{\partial z_0} = \frac{\partial \tilde{Q}}{\partial z_1} = 0 \right\} \subseteq \mathbb{P}^1 \times \mathbb{P}^1 \times \mathbb{P}^1\]
is empty.
\end{lemma}

\begin{proof}
Suppose \(([x_0:x_1],[y_0:y_1],[z_0:z_1]) \in \mathrm{Sing}(\tilde{Q})\).

If \(y_0y_1z_0z_1=0\), then
\[(x_0^2+x_1^2)(y_0^2+y_1^2)=(x_0^2+x_1^2)(z_0^2+z_1^2)=(y_0^2+y_1^2)(z_0^2+z_1^2)=0\]
implies that exactly one of \(y_0y_1=0\) or \(z_0z_1=0\) is true--so that also \(x_0^2+x_1^2=0\) and \(x_0x_1 \neq 0\); but then \(3x_0-2x_1=0\) gives a contradiction. 

From \(y_0y_1z_0z_1 \neq 0\), it follows that \((y_0^2 + y_1^2)(z_0^2 + z_1^2) \neq 0\). Also, if \(x_0^2+x_1^2=0\), then \(3x_0-2x_1=0\) again gives a contradiction. Thus
\[y_0^2-y_1^2 = z_0^2 - z_1^2 =0\]
implies
\[8x_0 \pm 3x_1 = 3x_0 + (8 \mp 4)x_1 = 0,\]
a contradiction which leaves open no further possibilities.
\end{proof}

Lemma \ref{SpecialK3} shows that \(\tilde{S}\) is a K3 surface; it is a variant of a K3 surface studied in \cite{McM} and \cite{Row}. The set of all curves parallel to axes contained in \(\tilde{S}\) is
\[\{C_1,\dots,C_{24}\} := \{C_{z,(i,0)},C_{z,(i,\infty)},C_{y,(0,i)},C_{y,(\infty,i)},C_{z,(2/3,i)},C_{z,(\infty,i)},C_{x,(i,0)},C_{x,(i,\infty)},\]
\[C_{y,(i,2/3)},C_{y,(i,\infty)},C_{x,(0,i)},C_{x,(\infty,i)},C_{z,(-i,0)},C_{z,(-i,\infty)},C_{y,(0,-i)},C_{y,(\infty,-i)},\]
\[C_{z,(2/3,-i)},C_{z,(\infty,-i)},C_{x,(-i,0)},C_{x,(-i,\infty)},C_{y,(-i,2/3)},C_{y,(-i,\infty)},C_{x,(0,-i)},C_{x,(\infty,-i)}\}.\]
Clearly, \(\tilde{S}\) is not pure. For example,
\[[C_{24}] = 2E_y+2E_z-2E_x-[C_7]-[C_8]-[C_{11}]-[C_{12}]-[C_{19}]-[C_{20}]-[C_{23}]\]
and
\[[C_{22}] = [C_{11}]+[C_{12}]-[C_{21}]+2E_x-2E_y-E_z+[C_7]+[C_8]+[C_{19}]+[C_{20}]\]
\[= -[C_{21}]+2E_x-2E_y-[C_9]-[C_{10}]+[C_7]+[C_8]+[C_{19}]+[C_{20}].\]

Set
\[\Gamma^+(\tilde{S}) := \{\mathcal{L} \in \Pic(\tilde{S})|(\mathcal{L} \cdot \mathcal{L})=-2 \text{ and } \mathcal{L} \text{ is effective}\}.\]
So
\[\Gamma^+(\tilde{S}) \cap \{\mathcal{L} + \mathcal{L}'|\{\mathcal{L},\mathcal{L}'\} \subseteq \Gamma^+(\tilde{S})\} \subseteq \Gamma^+(\tilde{S}),\]
and every \(\mathcal{L} \in \Pic(\tilde{S})\) satisfying \((\mathcal{L} \cdot \mathcal{L})=-2\) also satisfies \(|\{\mathcal{L},-\mathcal{L}\} \cap \Gamma^+(\tilde{S})|=1\).

\begin{proposition}\label{NefEmbed}
There is a lattice embedding \(L_{k,l,m} \leq \Pic(\tilde{S})\) such that \(\{B_1,B_2,B_3\} = \{E_x,E_y,E_z\}\). Thus setting \(\Gamma^+ := \Gamma \cap \Gamma^+(S)\) is an allowable choice that gives \((B_j \cdot \gamma) \geq 0\) for each \(B_j\) and every \(\gamma \in \Gamma^+\).
\end{proposition}

\begin{proof}
Since \((k,l,m) \in \mathcal{N}\), at least one of the lattice embeddings \(L_{k,l,m} \leq L_{6,0,0}\), \(L_{k,l,m} \leq L_{5,1,1}\), \(L_{k,l,m} \leq L_{4,2,2}\), or \(L_{k,l,m} \leq L_{3,3,3}\) exists; \(L_{6,0,0}\) is isometric to
\[\langle E_x,E_y,E_z,[C_7],[C_8],[C_{11}],[C_{12}],[C_{19}],[C_{20}] \rangle,\]
\(L_{5,1,1}\) is isometric to
\[\langle E_x,E_y,E_z,[C_2],[C_7],[C_8],[C_{11}],[C_{19}],[C_{20}],[C_{21}] \rangle,\]
\(L_{4,2,2}\) is isometric to
\[\langle E_x,E_y,E_z,[C_1],[C_2],[C_7],[C_8],[C_9],[C_{10}],[C_{19}],[C_{20}] \rangle,\]
and \(L_{3,3,3}\) is isometric to
\[\langle E_x,E_y,E_z,[C_1],[C_2],[C_7],[C_8],[C_9],[C_{10}],[C_{13}],[C_{19}],[C_{21}] \rangle.\]
Since \(E_x\), \(E_y\), and \(E_z\) are all nef, each \(B_j\) satisfies \((B_j \cdot \gamma) \geq 0\) for every \(\gamma \in \Gamma^+\).
\end{proof}

\subsection{Primitive embeddings of pure Picard lattices}

Let \(L_2\) be the lattice of rank 2 given by the matrix
\[M_2 := \left( \begin{array}{cc}
0 & 1 \\
1 & 0
\end{array} \right),\]
let \(L_8\) be the lattice of rank 8 given by the matrix
\[M_8 := \left( \begin{array}{cccccccc}
-2 & 0 & 0 & 1 & 0 & 0 & 0 & 0 \\
0 & -2 & 1 & 0 & 0 & 0 & 0 & 0 \\
0 & 1 & -2 & 1 & 0 & 0 & 0 & 0 \\
1 & 0 & 1 & -2 & 1 & 0 & 0 & 0 \\
0 & 0 & 0 & 1 & -2 & 1 & 0 & 0 \\
0 & 0 & 0 & 0 & 1 & -2 & 1 & 0 \\
0 & 0 & 0 & 0 & 0 & 1 & -2 & 1 \\
0 & 0 & 0 & 0 & 0 & 0 & 1 & -2
\end{array} \right),\]
and set \(L_{K3} := (L_2)^{\oplus 3} \oplus (L_8)^{\oplus 2}\); so \(L_{K3}\) has rank 22, is even in the sense that every element of \(L_{K3}\) has even self-intersection, and is unimodular in the sense that \(M_{K3} := (M_2)^{\oplus 3} \oplus (M_8)^{\oplus 2}\) is invertible over \(\mathbb{Z}\). For any complex K3 surface \(T\), it is a well-known fact (e.g.,\cite{May},\cite{bhpv},\cite{McM}) that the cup product makes \(H^2(T,\mathbb{Z})\) into a lattice isometric to \(L_{K3}\). A lattice embedding \(L \leq L'\) is said to be primitive if \((L^\perp)^\perp = L\) (where the orthogonal lattices are taken in \(L'\)) or, equivalently, if \((L \otimes \mathbb{Q}) \cap L' = L\). For example, by the Lefschetz theorem on (1,1) classes (e.g.,\cite{bhpv}), \(\Pic(T) \leq H^2(T,\mathbb{Z})\) is a primitive lattice embedding for every complex K3 surface \(T\).

For \((k,l,m) \in \mathcal{N}\), we have established that \(L_{k,l,m}\) can be assigned a nef cone which contains every \(B_j\), and furthermore that any effective isometry between \(L_{k,l,m}\) and the Picard lattice of a K3 surface then forces the K3 surface to be pure of type \((k,l,m)\). To prove the existence of pure K3 surfaces of type \((k,l,m)\), it remains to be shown only that \(L_{k,l,m}\) embeds primitively in \(L_{K3}\).

\begin{proposition}\label{K3exist}
If \((k,l,m) \neq (3,3,3)\), then there is a primitive lattice embedding \(L_{k,l,m} \leq L_{K3}\).
\end{proposition}

\begin{proof}
Since the natural embedding of \(L_{k,l,m}\) into one of \(L_{6,0,0}\), \(L_{5,1,1}\), \(L_{4,2,2}\), or \(L_{3,3,2}\) has a basis which is a subset of a basis for the larger lattice, it must be primitive. So \(L_{k,l,m}\) has a primitive embedding in \(L_{K3}\) if \(L_{6,0,0}\), \(L_{5,1,1}\), \(L_{4,2,2}\), and \(L_{3,3,2}\) do.

Let \(\{\beta_1,\dots,\beta_{22}\}\) be a basis for \(L_{K3}\) in which \(M_{K3}\) gives \((\_ \cdot \_)_{L_{K3}}\). Set
\[\mathcal{B}_{6,0,0} =\]
\[\{\beta_1 + 2\beta_2 + \beta_4 + \beta_6 + \beta_{10} + \beta_{18},\beta_3 + \beta_2 + \beta_6,\beta_5 + \beta_2 + \beta_4,\]
\[\beta_7,\beta_9,\beta_{11},\beta_{15},\beta_{17},\beta_{19}\},\]
\[\mathcal{B}_{5,1,1} =\]
\[\{\beta_1 + 2\beta_2 + \beta_4 + \beta_6 + \beta_{10} + \beta_{18},\beta_3 + \beta_4 + \beta_2 + \beta_6 + \beta_{13},\beta_5 + \beta_6 + \beta_2 + \beta_4 + \beta_{21},\]
\[\beta_7,\beta_9,\beta_{11},\beta_{14},\beta_{15},\beta_{17},\beta_{22}\},\]
\[\mathcal{B}_{4,2,2} =\]
\[\{\beta_1 + 2\beta_2 + \beta_4 + \beta_6 + \beta_{10} + \beta_{18},\beta_3 + \beta_4 + \beta_2 + \beta_6 + \beta_{13},\beta_5 + \beta_6 + \beta_2 + \beta_4 + \beta_{21},\]
\[\beta_7,\beta_9,\beta_{12},\beta_{14},\beta_{15},\beta_{17},\beta_{20},\beta_{22}\},\]
and
\[\mathcal{B}_{3,3,2} =\]
\[\{\beta_1 + \beta_2 + \beta_4 + \beta_6 + \beta_{10},\beta_3 + \beta_4 + \beta_2 + \beta_6 + \beta_{18},\beta_5 + 2\beta_6 + \beta_2 + \beta_4 + \beta_{13} + \beta_{21},\]
\[\beta_7,\beta_9,\beta_{11},\beta_{14},\beta_{15},\beta_{17},\beta_{19},\beta_{22}\},\]
Since the matrices which send \(\{\beta_1,\beta_3,\beta_5\}\) to the first three entries of \(\mathcal{B}_{6,0,0}\), \(\mathcal{B}_{5,1,1}\), \(\mathcal{B}_{4,2,2}\), and \(\mathcal{B}_{3,3,2}\) and fix the remaining \(\beta_j\) are all invertible over \(\mathbb{Z}\), \(\mathcal{B}_{6,0,0}\), \(\mathcal{B}_{5,1,1}\), \(\mathcal{B}_{4,2,2}\), and \(\mathcal{B}_{3,3,2}\) are all subsets of bases for \(L_{K3}\); so they generate primitive embeddings of \(L_{6,0,0}\), \(L_{5,1,1}\), \(L_{4,2,2}\), and \(L_{3,3,2}\) in \(L_{K3}\).
\end{proof}

\subsection{Contradictions in pure Picard lattices of high rank}

Fix an ordered triple \((k,l,m)\) of non-negative integers such that \((k,l,m) \notin \mathcal{N}\). Up to reordering, one of \((k,l,m) \geq (7,0,0)\), \((k,l,m) \geq (6,1,0)\), \((k,l,m) \geq (5,2,0)\), or \((k,l,m) \geq (4,3,0)\) is true. Taking \(\tilde{S} \subseteq \mathbb{P}^1 \times \mathbb{P}^1 \times \mathbb{P}^1\) as above, we will use the arrangement of the curves parallel to axes in \(\tilde{S}\) to show that there is no pure K3 surface whose Picard lattice is isometric to \(L_{k,l,m}\). 

\begin{proposition}\label{LargePureExtraC}
There is no pure K3 surface of type \((k,l,m)\) in \(\mathbb{P}^1 \times \mathbb{P}^1 \times \mathbb{P}^1\).
\end{proposition}

\begin{proof}
Since \(L_{7,0,0}\) is isometric to
\[\langle E_x,E_y,E_z,[C_7],[C_8],[C_{11}],[C_{12}],[C_{19}],[C_{20}],[C_{23}] \rangle,\]
which contains \([C_{24}]\); \(L_{6,1,0}\) is isometric to
\[\langle E_x,E_y,E_z,[C_7],[C_8],[C_{11}],[C_{12}],[C_{19}],[C_{20}],[C_{21}] \rangle,\]
which contains \([C_{22}]\); \(L_{5,2,0}\) is isometric to
\[\langle E_x,E_y,E_z,[C_7],[C_8],[C_{11}],[C_{19}],[C_{20}],[C_{21}],[C_{22}] \rangle,\]
which contains \([C_{12}]\); and \(L_{4,3,0}\) is isometric to
\[\langle E_x,E_y,E_z,[C_7],[C_8],[C_9],[C_{10}],[C_{19}],[C_{20}],[C_{21}] \rangle,\]
which contains \([C_{22}]\); each of these lattices contains an element \(\gamma_0\) which satisfies \((\gamma_0 \cdot \gamma_0)=-2\), \((\gamma_0 \cdot E_{\omega'})=1\) for some \(E_{\omega'}\), \((\gamma_0 \cdot E_{\omega \neq \omega'})=0\), and \((\gamma_0 \cdot [C_j]) \geq 0\) for all \([C_j]\) in the given basis.

Suppose \(S \subseteq \mathbb{P}^1 \times \mathbb{P}^1 \times \mathbb{P}^1\) is pure of type \((k,l,m)\); so, in light of the natural embedding of one the lattices listed above in \(L_{k,l,m}\), there must be \(\gamma_0 \in \Pic(S)\) with the properties described above and, moreover, the property that \(\gamma_0 \notin \mathcal{B}(S)\). Since \(\gamma_0\) is effective and is in \(\langle E_{\omega_1},E_{\omega_2} \rangle^\perp\) for some distinct \(E_{\omega_1}\) and \(E_{\omega_2}\), it is a sum \(\gamma_0 = [D_1]+\dots+[D_n]\) of (a priori not necessarily distinct) classes of prime divisors all satisfying \((D_j \cdot D_j)=-2\) and \(D_j \in \langle E_{\omega_1},E_{\omega_2} \rangle^\perp\). Then, as in the proof of Proposition \ref{K3embed}, each \(D_j\) must be a curve parallel to an axis--which leads to a contradiction.
\end{proof}

\subsection{Proof of Theorem \ref{Existence}}\label{FINDpf}

Suppose \((k,l,m) \in \mathcal{N}-\{(3,3,3)\}\). By Proposition \ref{K3exist}, there is a primitive lattice embedding \(L_{k,l,m} \leq L_{K3}\). Since \(L_{k,l,m}^\perp\) has signature \((2,17-k-l-m)\), \(L_{k,l,m}^\perp \otimes \mathbb{R}\) contains a positive-definite two-dimensional subspace \(V\) such that \(V^\perp \cap L_{K3} = L_{k,l,m}\). Thus the surjectivity of the period map for K3 surfaces (e.g.,\cite{bhpv},\cite{Dol}) implies (with an application of the Leschetz theorem on (1,1) classes) that there is a K3 surfaces \(S\) with \(\Pic(S)\) isometric to \(L_{k,l,m}\). Moreover, the isometry between \(\Pic(S)\) and \(L_{k,l,m}\) can be taken to be effective for any allowable choice of \(\Gamma^+\). So, by Propositions \ref{NefEmbed} and \ref{K3embed}, there is a pure K3 surface of type \((k,l,m)\). In fact, since it is established that at least one exists, the moduli space \(\mathcal{M}(L_{k,l,m})\) of ample \(L_{k,l,m}\)-polarized K3 surfaces (with \(\Gamma^+\) fixed) (e.g.,\cite{Dol},\cite{bhpv}) is a quasi-projective variety of dimension \(17-k-l-m\). For every \(T \in \mathcal{M}(L_{k,l,m})\), there is an effective primitive lattice embedding \(L_{k,l,m} \leq \Pic(T)\); so either \(T\) is pure of type \((k,l,m)\) or \(\rho(T) > 3+k+l+m\) and \(T \in \mathcal{M}(\Pic(T))\). Since there are only countably many possible such \(\Pic(T)\) not effectively isometric to \(L_{k,l,m}\) and the dimension of \(\mathcal{M}(\Pic(T))\) for each of these is less than \(17-k-l-m\), the space \(\mathcal{M}_0(L_{k,l,m})\) of K3 surfaces \(S\) with \(\Pic(S)\) effectively isometric to \(L_{k,l,m}\) is very general in \(\mathcal{M}(L_{k,l,m})\). By Propositions \ref{NefEmbed} and \ref{K3embed}, \(\mathcal{M}_0(L_{k,l,m})\) is the space of isomorphism classes of K3 surfaces contained in \(\mathcal{U}_{k,l,m}\).

Let \(\mathcal{V}_{k,l,m} \subseteq \mathbb{P}^{26}\) denote the space of all effective divisors of tri-degree \((2,2,2)\) whose supports contain some union of curves
\[C_{x,1} \cup \dots \cup C_{x,k} \cup C_{y,1} \cup \dots \cup C_{y,l} \cup C_{z,1} \cup \dots \cup C_{z,m}\]
such that each \(C_{\omega,j}\) is a curve parallel to the \(\omega\)-axis and any two distinct \(C_{\omega,j}\) and \(C_{\omega',j'}\) are disjoint, and let \(\mathcal{I}_{k,l,m}\) denote the incidence variety in
\[\mathbb{P}^{26} \times (\mathbb{P}^1 \times \mathbb{P}^1)^{k+l+m} = \{(Q,[\alpha_{x,1}:\beta_{x,1}],[\delta_{x,1}:\epsilon_{x,1}],\dots,[\alpha_{z,m}:\beta_{z,m}],[\delta_{z,m}:\epsilon_{z,m}]\}\]
defined by
\[Q_{\omega,0}(\alpha_{\omega,j},\beta_{\omega,j},\delta_{\omega,j},\epsilon_{\omega,j}) = Q_{\omega,1}(\alpha_{\omega,j},\beta_{\omega,j},\delta_{\omega,j},\epsilon_{\omega,j}) = Q_{\omega,2}(\alpha_{\omega,j},\beta_{\omega,j},\delta_{\omega,j},\epsilon_{\omega,j}) = 0\]
for all \(\omega\) and \(j\). Since \(\mathcal{V}_{k,l,m}\) is the image under the projection to \(\mathbb{P}^{26}\) of a complement \(\mathcal{V}_{k,l,m}' \subseteq \mathcal{I}_{k,l,m}\) of finitely many sections from linear subspaces of \((\mathbb{P}^1 \times \mathbb{P}^1)^{k+l+m}\), it is a quasi-projective variety. For a fixed point \(\zeta \in (\mathbb{P}^1 \times \mathbb{P}^1)^{k+l+m}\), the equations defining \(\mathcal{I}_{k,l,m}\) show that the fiber over \(\zeta\) of the projection of \(\mathcal{I}_{k,l,m}\) to \((\mathbb{P}^1 \times \mathbb{P}^1)^{k+l+m}\) is a linear subspace of \(\mathbb{P}^{26}\) of codimension at most \(3(k+l+m) \leq 24\). Since the projection of \(\mathcal{V}_{k,l,m}'\) to \((\mathbb{P}^1 \times \mathbb{P}^1)^{k+l+m}\) is Zariski dense, it follows that \(\mathcal{V}_{k,l,m}\) is irreducible. By the construction of \(\mathcal{V}_{k,l,m}\), \(\mathcal{U}_{k,l,m}\) is very general in \(\mathcal{V}_{k,l,m}\). So the closure of \(\mathcal{U}_{k,l,m}\) contains \(\mathcal{V}_{k',l',m'}\) for all \((k',l',m') \in \mathcal{N}\) satisfying \((k,l,m) < (k',l',m')\).

The claim for \((k,l,m) \notin \mathcal{N}\) is given by Proposition \ref{LargePureExtraC}.
\hfill \qed

\section{Computing Entropies on Pure K3 Surfaces}\label{COMPUTE}

Fix \(S \in \mathcal{U}_{k,l,m}\) for some \((k,l,m) \in \mathcal{N}\). It is a well-known fact (e.g.,\cite{Bea}) that every birational self-map on \(S\) extends to an automorphism of \(S\). So, in particular, each \(\tau_\omega\)--and hence also \(f\)--defines an automorphism of \(S\).

\subsection{Cohomological actions of involutions}

We will compute the action of \(\tau_x^*\) on \(\Pic(S)\); the actions of \(\tau_y^*\) and \(\tau_z^*\) are similar. Write \(\mathcal{B}_x(S) = \{C_{x,p_1},\dots,C_{x,p_k}\}\).

\begin{proposition}\label{fMatrices}
Each \([C_{x,p_j}]\) is fixed by \(\tau_x^*\), as are \(E_y\) and \(E_z\). For each \([C_{y,p}] \in \mathcal{B}(S)\),
\[\tau_x^*[C_{y,p}] = E_z - [C_{y,p}].\]
For each \([C_{z,p}] \in \mathcal{B}(S)\),
\[\tau_x^*[C_{z,p}] = E_y - [C_{z,p}].\]
Finally,
\[\tau_x^*E_x = -E_x + 2E_y + 2E_z - [C_{x,p_1}] - \dots -[C_{x,p_k}].\]
\end{proposition}

\begin{proof}
Since \(\tau_x = \tau_x^{-1}\) preserves every elliptic curve which is a fiber over either the \(y\)-axis or the \(z\)-axis, \(\tau_X^*\) must fix \(E_y\) and \(E_z\). For \(E_{\omega=\alpha}\) containing a curve \(C\) parallel to an axis, \(E_{\omega=\alpha} - C\) is an effective divisor of bi-degree \((1,2)\) or \((2,1)\). It follows from Remark \ref{NoD2} that in fact \(E_{\omega=\alpha} - C\) is a prime divisor not parallel to any axis. For each \(C_{x,p_j}\), write \(p_j=(\alpha,\delta)\); since \(\tau_x\) preserves both \(E_{y=\alpha}\) and \(E_{z=\delta}\), it must fix \(C_{x,p_j}\). For \([C_{y,p}] \in \mathcal{B}(S)\), write \(p = (\alpha,\delta)\); since \(\tau_x\) preserves \(E_{z=\alpha}\) and does not preserve \(C_{y,p}\), it must take \(C_{y,p}\) to \(E_{z=\alpha}-C_{y,p}\). It follows similarly that \(\tau_x^*\) takes \(C_{z,p}\) to \(E_{y=\delta}-C_{z,p}\) for \([C_{z,p}] \in \mathcal{B}(S)\).

With the action of \(\tau_x^*\) established for all elements of \(\mathcal{B}(S)\) except \(E_x\), the conditions that \(\tau_x\) is an involution and \(\tau_x^*\) preserves the intersection form given by \(M_{k,l,m}\) force the formula given for \(\tau_x^*E_x\) to hold.
\end{proof}

Proposition \ref{fMatrices} shows that the action of \(f^*\) in the basis \(\mathcal{B}(S)\) is constant on \(\mathcal{U}_{k,l,m}\), and gives the necessary information for computation of \(\lambda(f)\).

\begin{lemma}\label{fOrder}
If \((k',l',m')\) is a reordering of \((k,l,m)\), then \(\lambda(f)\) is constant on \(\mathcal{U}_{k,l,m} \cup \mathcal{U}_{k',l',m'}\).
\end{lemma}

\begin{proof}
Fix \(S' \in \mathcal{U}_{k',l',m'}\). Some
\[g \in \mathcal{G} := \{\tau_z \circ \tau_y \circ \tau_x,\tau_z \circ \tau_x \circ \tau_y,\tau_y \circ \tau_x \circ \tau_z,\tau_y \circ \tau_z \circ \tau_x,\tau_x \circ \tau_z \circ \tau_y,\tau_x \circ \tau_y \circ \tau_z\}\]
has the property that the action of \(g^*\) on \(\Pic(S')\) is essentially identical to the action of \(f^*\) on \(\Pic(S)\). Since every element of \(\mathcal{G}\) is conjugate (by some element in \(\langle \tau_x,\tau_y,\tau_z \rangle\)) to either \(f\) or \(f^{-1}\), the spectral radius of \(f^*\) on \(\Pic(S')\) is the same as that of \(f^*\) on \(\Pic(S)\).
\end{proof}

\subsection{Proof of Theorem \ref{Entropies}}

By Proposition \ref{fMatrices} and Lemma \ref{fOrder}, the action of \(f^*\) on \(\Pic(S)\) depends only on the unordered type of \(S\). The following table gives the spectral radius (computed in Mathematica) of \(f^*\) for all types of \(S\), and one can check that \(\lambda(k,l,m)\) exhibits the claimed behavior.
\hfill \qed

\vspace{.2in}

\begin{center}
\begin{tabular}{c c c}
\hline
\phantom{min. poly. for \(\lambda(f)\) plus} & \phantom{min. poly. for \(\lambda(f)\) plus} & \phantom{min. poly. for \(\lambda(f)\) plus} \\
\((k,l,m)\) & \(\lambda(f)\) & min. poly. for \(\lambda(f)\) \\
\phantom{min. poly. for \(\lambda(f)\) plus} & \phantom{min. poly. for \(\lambda(f)\) plus} & \phantom{min. poly. for \(\lambda(f)\) plus} \\
\hline
\((0,0,0)\) & 17.944... & \(t^2 - 18t + 1\)  \\
\hline
\((1,0,0)\) & 15.937... & \(t^2 - 16t + 1\) \\
\hline
\((2,0,0)\) & 13.928... & \(t^2 - 14t + 1\) \\
\hline
\((3,0,0)\) & 11.916... & \(t^2 - 12t + 1\) \\
\hline
\((4,0,0)\) & 9.898... & \(t^2 - 10t + 1\) \\
\hline
\((5,0,0)\) & 7.872... & \(t^2 - 8t + 1\) \\
\hline
\((6,0,0)\) & 5.828... & \(t^2 - 6t + 1\) \\
\hline
\((1,1,0)\) & 14.011... & \(t^4 - 16t^3 + 29t^2 - 16t + 1\) \\
\hline
\((2,1,0)\) & 12.113... & \(t^4 - 14t^3 + 24t^2 - 14t + 1\) \\
\hline
\((3,1,0)\) & 10.261... & \(t^4 - 12t^3 + 19t - 12t + 1\) \\
\hline
\((4,1,0)\) & 8.487... & \(t^4 - 10t^3 + 14t - 10t + 1\) \\
\hline
\((5,1,0)\) & 6.854... & \(t^2 - 7t + 1\) \\
\hline
\((2,2,0)\) & 10.375... & \(t^4 - 12t^3 + 18t^2 - 12t + 1\) \\
\hline
\((3,2,0)\) & 8.758... & \(t^4 - 10t^3 +12t^2 - 10t^3 + 1\) \\
\hline
\((4,2,0)\) & 7.327... & \(t^4 - 8t^3 + 6t^2 - 8t + 1\) \\
\hline
\((3,3,0)\) & 7.471... & \(t^4 - 8t^3 + 5t^2 - 8t + 1\) \\
\hline
\((1,1,1)\) & 12.113... & \(t^4 - 14t^3 + 24t^2 - 14t + 1\) \\
\hline
\((2,1,1)\) & 10.261... & \(t^4 - 12t^3 + 19t - 12t + 1\) \\
\hline
\((3,1,1)\) & 8.487... & \(t^4 - 10t^3 + 14t - 10t + 1\) \\
\hline
\((4,1,1)\) & 6.854... & \(t^2 - 7t +1\) \\
\hline
\((5,1,1)\) & 5.462... & \(t^4 - 6t^3 + 4t^2 - 6t + 1\) \\
\hline
\((2,2,1)\) & 8.487... & \(t^4 - 10t^3 + 14t - 10t + 1\) \\
\hline
\((3,2,1)\) & 6.854... & \(t^2 - 7t + 1\) \\
\hline
\((4,2,1)\) & 5.462... & \(t^4 - 6t^3 + 4t^2 - 6t + 1\) \\
\hline
\((3,3,1)\) & 5.462... & \(t^4 - 6t^3 + 4t^2 - 6t + 1\) \\
\hline
\((2,2,2)\) & 6.678... & \(t^4 - 8t^3 + 10t^2 - 8t + 1\) \\
\hline
\((3,2,2)\) & 5.037... & \(t^4 - 6t^3 + 6t^2 - 6t + 1\) \\
\hline
\((4,2,2)\) & 3.732... & \(t^2 - 4t + 1\) \\
\hline
\end{tabular}
\clearpage
\begin{tabular}{c c c }
\phantom{min. poly. for \(\lambda(f)\) plus} & \phantom{min. poly. for \(\lambda(f)\) plus} & \phantom{min. poly. for \(\lambda(f)\) plus} \\
\hline
\((3,3,2)\) & 3.441... & \(t^4 - 4t^3 + 3t^2 - 4t +1\) \\
\hline
\((3,3,3)\) & 1 & \(t-1\) \\
\hline\hline
\end{tabular}
\end{center}

\bibliographystyle{abbrv}
\bibliography{refs-2015.06.20}

\end{document}